\documentclass[reqno,11pt]{amsart}
\usepackage{amsmath, amsthm, amssymb, booktabs}
\usepackage{mathrsfs}
\usepackage{color}

\advance \topmargin by -\headheight
\advance \topmargin by -\headsep
     
\evensidemargin \oddsidemargin
\newtheorem{theorem}{Theorem}[section]
\newtheorem{lemma}[theorem]{Lemma}

\theoremstyle{definition}

\theoremstyle{remark}

\numberwithin{equation}{section}

\newcommand{\mmod}[1]{\,\,(\text{mod}\,\,#1)}

\def\bfa{{\mathbf a}}

\def\bfc{{\mathbf c}}

\def\bfx{{\mathbf x}}
\def\bfy{{\mathbf y}}
\def\bfz{{\mathbf z}}

\def\calC{{\mathscr C}}

\def\calI{{\mathscr I}}

 \def\Ktil{{\widetilde K}}

\def\Ktil{\widetilde K}

\def\dbN{{\mathbb N}}
\def\dbR{{\mathbb R}}
\def\dbZ{{\mathbb Z}}\def\dbQ{{\mathbb Q}}

\def\grA{{\mathfrak A}}
\def\grB{{\mathfrak B}}

\def\grf{{\mathfrak f}}\def\grF{{\mathfrak F}}

\def\grJ{{\mathfrak J}}
\def\grk{{\mathfrak k}}
\def\grL{{\mathfrak L}}
\def\grm{{\mathfrak m}}\def\grM{{\mathfrak M}}\def\grN{{\mathfrak N}}
\def\grn{{\mathfrak n}}\def\grS{{\mathfrak S}}\def\grP{{\mathfrak P}}
\def\grW{{\mathfrak W}}\def\grB{{\mathfrak B}}
\def\grK{{\mathfrak K}}\def\grL{{\mathfrak L}}\def\grp{{\mathfrak p}}
\def\grT{{\mathfrak T}}

\def\grW{{\mathfrak W}}

\def\alp{{\alpha}} \def\bfalp{{\boldsymbol \alpha}} 
\def\bet{{\beta}}  \def\bfbet{{\boldsymbol \beta}}
\def\gam{{\gamma}}  \def\bfgam{{\boldsymbol \gamma}}
\def\bfgamtil{{\widetilde \bfgam}} \def\gamtil{{\widetilde \gam}}
\def\del{{\delta}} \def\Del{{\Delta}}

\def\tet{{\theta}} \def\bftet{{\boldsymbol \theta}} 

 \def\Lam{{\Lambda}}

\def\sig{{\sigma}}  

\def\Ups{{\Upsilon}} 

 \def\Ome{{\Omega}} 
\def\d{{\partial}}
\def\eps{\varepsilon}

\def\le{\leqslant} \def\ge{\geqslant}

\def\d{{\,{\rm d}}}

\makeatletter
\@namedef{subjclassname@2020}{\textup{2020} Mathematics Subject Classification}
\makeatother

\begin{document}
\title[Subconvexity via the circle method]{Rational lines on diagonal hypersurfaces and 
subconvexity via the circle method}
\author[Trevor D. Wooley]{Trevor D. Wooley}
\address{Department of Mathematics, Purdue University, 150 N. University Street, West 
Lafayette, IN 47907-2067, USA}
\email{twooley@purdue.edu}
\subjclass[2010]{11D45, 11D72, 11P55}
\keywords{Hardy-Littlewood method, Diophantine equations, rational lines}
\thanks{The author is supported by NSF grants DMS-1854398 and DMS-2001549}
\begin{abstract} Fix $k,s,n\in \dbN$, and consider non-zero integers $c_1,\ldots ,c_s$, not 
all of the same sign. Provided that $s\ge k(k+1)$, we establish a Hasse principle for the 
existence of lines having integral coordinates lying on the affine diagonal hypersurface 
defined by the equation $c_1x_1^k+\ldots +c_sx_s^k=n$. This conclusion surmounts the 
conventional convexity barrier tantamount to the square-root cancellation limit for this 
problem.
\end{abstract}
\maketitle

\section{Introduction} The investigation of rational linear spaces on algebraic varieties was 
pursued by Brauer \cite{Bra1945} and Birch \cite{Bir1957} as a key step in their inductive 
strategies for establishing the existence of rational points on complete intersections. This 
initial work in the middle of the last century has more recently evolved, in contributions of 
Parsell \cite{Par2000, Par2009} and Brandes \cite{Bra2014}, to encompass quantitative 
considerations. In this paper, we also investigate the abundance of rational lines, but now 
on affine diagonal hypersurfaces. By applying the Hardy-Littlewood (circle) method, we 
derive a certain Hasse principle for the existence of lines having integral coordinates lying 
on the hypersurface. A notable feature of our application is that it goes beyond the 
convexity limit of the circle method, by which we mean the square-root barrier that 
ordinarily restricts the method to problems in which the number of available variables 
exceeds twice the inherent degree.\par

In order to describe our conclusions more precisely, we must introduce some notation. We 
fix natural numbers $k\ge 2$ and $s$, and we consider the affine hypersurface defined by 
the diagonal equation
\begin{equation}\label{1.1}
c_1x_1^k+\ldots +c_sx_s^k=n,
\end{equation}
in which $c_i\in \dbZ\setminus \{0\}$ $(1\le i \le s)$ and $n\in \dbZ\setminus \{0\}$ are 
fixed. We assume, in particular, that the coefficients $c_i$ are neither all positive nor all 
negative. For each exponent $k$, a conventional application of the circle method confirms 
the existence of a positive number $s_0(k)$ having the property that the solutions of the 
equation \eqref{1.1} satisfy the weak approximation property provided only that 
$s\ge s_0(k)$. Indeed, it follows from the work and methods of earlier scholars that when 
$2\le k\le 15$ one has $s_0(k)\le t_0(k)$, where $t_0(k)$ is defined according to Table 
\ref{tab1} below (see \cite{BW2022, Est1962, Vau1989a, Vau1989b, VW1994, VW1995, 
Woo2016} for the necessary ideas). Meanwhile, recent work of the author with Br\"udern 
\cite{BW2022} may be routinely applied to confirm that 
$s_0(k)\le \lceil k(\log k+4.20032)\rceil$. Moreover, subject to real and $p$-adic solubility 
hypotheses, it follows under the same conditions that the equation \eqref{1.1} possesses an 
abundance of integral solutions in which, for each $r\ge 1$, there is no $r$-tuple 
$(i_1,i_2,\ldots ,i_r)$ of indices with $1\le i_1<i_2<\ldots <i_r\le s$ for which
\[
c_{i_1}x_{i_1}^k+c_{i_2}x_{i_2}^k+\ldots +c_{i_r}x_{i_r}^k=0.
\]
Henceforth, we refer to the latter as the condition that there be {\it no vanishing subsums}, 
and we note in particular that it implies that no variable $x_i$ is equal to $0$.\par

\begin{table}[h]
\begin{center}
\begin{tabular}{ccccccccccccccccc}
\toprule
$k$ & $2$ & $3$ & $4$ & $5$ & $6$ & $7$ & $8$ & $9$ & $10$ & $11$ & $12$ & $13$ & 
$14$ & $15$\\
$t_0(k)$ & $4$ & $7$ & $12$ & $17$ & $24$ & $31$ & $39$ & $47$ & $55$ & $63$ & 
$72$ & $81$ & $89$ & $97$ \\
\bottomrule
\end{tabular}\\[6pt]
\end{center}
\caption{Upper bounds for $s_0(k)$ when $2\le k\le 15$.}\label{tab1}
\end{table}

Our interest in this paper lies with the existence of linear solution spaces of the equation 
\eqref{1.1} of the shape $\bfx=\bfy+t\bfz$, with $\bfy\in (\dbZ\setminus \{0\})^s$ and 
$\bfz\in \dbZ^s\setminus\{ {\bf 0}\}$. Subject to local solubility conditions and the 
hypothesis $s\ge s_0(k)$, it follows from the above discussion that there exists an $s$-tuple 
$\bfy\in (\dbZ\setminus \{0\})^s$ satisfying the equation
\begin{equation}\label{1.2}
c_1y_1^k+\ldots +c_sy_s^k=n.
\end{equation}
With this solution fixed, we denote by $N_{s,k}(B;\bfy)$ the number of integral $s$-tuples 
$\bfz\in \dbZ^s\cap [-B,B]^s$ for which the equation
\begin{equation}\label{1.3}
c_1(y_1+tz_1)^k+\ldots +c_s(y_s+tz_s)^k=n
\end{equation}
holds identically as a polynomial in $t$. By expanding the powers in \eqref{1.3} via the 
binomial theorem, and recalling \eqref{1.2}, one sees that the condition on these $s$-tuples 
is equivalent to insisting that $\bfz$ satisfy the system of equations
\begin{equation}\label{1.4}
\sum_{i=1}^sc_iy_i^{k-j}z_i^j=0\quad (1\le j\le k).
\end{equation}
We choose in this paper to focus on the situation with $k\ge 3$. The situation with $k=1$ is 
a matter for linear algebra, while that with $k=2$ is accessible to the theory of quadratic 
forms. Indeed, by eliminating a variable between the linear and quadratic equations in 
\eqref{1.4}, one sees that the problem of determining $N_{s,2}(B;\bfy)$ is equivalent to 
the classical problem of counting integral solutions of a homogeneous quadratic equation in 
$s-1$ variables subject to a congruence condition, and this is well-understood for all $s$. 

\begin{theorem}\label{theorem1.1}
Let $s$ and $k$ be natural numbers with $k\ge 3$ and $s\ge k(k+1)$. Also, let 
$c_1,\ldots ,c_s$ and $n$ be fixed non-zero integers, with $c_1,\ldots ,c_s$ neither all 
positive nor all negative. Suppose that $y_1,\ldots ,y_s$ are non-zero integers satisfying the 
equation \eqref{1.2}. Then, provided that the system \eqref{1.4} has non-singular real and 
$p$-adic solutions for every prime number $p$, there is a positive number 
$\calC_{s,k}(\bfy)$ for which
\begin{equation}\label{1.5}
N_{s,k}(B;\bfy)=\calC_{s,k}(\bfy)B^{s-k(k+1)/2}+o(B^{s-k(k+1)/2}).
\end{equation}
\end{theorem}

Some remarks are in order concerning the nature of the conclusion provided by Theorem 
\ref{theorem1.1}. First, since $s_0(k)\le k(k+1)$ for all natural numbers $k$, the discussion 
above ensures that there are plenty of solutions $\bfy\in (\dbZ\setminus \{0\})^s$ 
satisfying the equation \eqref{1.2} whenever local solubility conditions permit such a 
conclusion. Here, it is apparent that obstructions to $p$-adic solubility may be present when 
the bulk of the coefficients $c_i$ are divisible by $p$, and yet $n$ is not. However, one may 
regard the first important hypothesis of this theorem as being essentially harmless. Next, as 
we demonstrate in \S7, the existence of non-singular real and $p$-adic solutions of the 
system \eqref{1.4} follows in two simple circumstances occurring generically. First, should 
the solution $\bfy\in (\dbZ\setminus \{0\})^s$ of the equation \eqref{1.2} satisfy the 
condition that there be no vanishing subsums, then any solution $\bfz\ne {\mathbf 0}$ of 
the system \eqref{1.4} over $\dbR$ or $\dbQ_p$ is automatically non-singular. Secondly, 
subject only to the condition that $y_1,\ldots ,y_s$ are non-zero integers satisying the 
equation \eqref{1.2}, any solution $\bfz$ of the system \eqref{1.4} over $\dbR$ or 
$\dbQ_p$ is non-singular whenever $z_i\ne 0$ for $1\le i\le s$.\par

Finally, as the reader will have anticipated, one may interpret the coefficient 
$\calC_{s,k}(\bfy)$ appearing in the asymptotic formula \eqref{1.5} as a product of local 
densities, the description of which requires some preparation. When $p$ is a prime number 
and $h\in \dbN$, write $M_p(h)$ for the number of solutions $\bfz$ of the system 
\eqref{1.4} with $\bfz \in (\dbZ/p^h\dbZ)^s$. Also, when $\eta>0$, denote by 
$M_\infty (\eta)$ the volume of the subset of $[-1,1]^s$ defined by the inequalities
\[
\biggl| \sum_{i=1}^s c_iy_i^{k-j}z_i^j\biggr| <\eta \quad (1\le j\le k).
\]
Then the limits
\[
\sig_\infty=\lim_{\eta\rightarrow 0+}(2\eta)^{-k}M_\infty (\eta) \quad \text{and}\quad 
\sig_p=\lim_{h\rightarrow \infty}p^{h(k-s)}M_p(h),
\]
when they exist, respectively define the real and $p$-adic densities of solutions of the 
system \eqref{1.4}. We show in \S5 that, under the hypotheses of the statement of 
Theorem \ref{theorem1.1}, both limits exist, and one has 
$\calC_{s,k}(\bfy)=\sig_\infty \prod_p\sig_p$, where the product is taken over all prime 
numbers $p$. Furthermore, one has $1\ll \calC_{s,k}(\bfy)\ll 1$.\par

The conclusion of Theorem \ref{theorem1.1} shows, subject to natural local solubility 
conditions and the constraint $s\ge k(k+1)$, that there is an abundance of affine lines 
having integral coefficients passing through each eligible integral point of the hypersurface 
determined by the equation \eqref{1.1}. In the situation wherein $s=k(k+1)$, the 
conclusion of Theorem \ref{theorem1.1} surmounts the convexity barrier in the circle 
method, since the number of variables is precisely twice the sum of the degrees of the 
polynomials defining the system of equations \eqref{1.4}. This subconvexity conclusion is 
made more apparent by a consideration of the associated exponential sums. When $k\ge 3$ 
and $X$ is a large real number, define $f(\bfalp)=f_k(\bfalp;X)$ by putting
\begin{equation}\label{1.6}
f_k(\bfalp;X)=\sum_{|x|\le X}e(\alp_1x+\alp_2x^2+\ldots +\alp_kx^k),
\end{equation}
where, as usual, we write $e(z)=e^{2\pi iz}$. We introduce a Hardy-Littlewood dissection 
to facilitate discussion. Write $L=X^{1/(8k^2)}$. Then, when
\[
1\le q\le L,\quad 0\le \bfa\le q\quad \text{and}\quad (q,\bfa)=1,
\]
we define the major arc $\grP(q,\bfa)$ by
\[
\grP(q,\bfa)=\{\bfalp\in [0,1)^k:|\alp_j-a_j/q|\le LX^{-j}\ (1\le j\le k)\}.
\]
Here and throughout this paper, we facilitate concision by adopting the use of extended 
vector notation. Thus, we write $0\le \bfa\le q$ to denote that $0\le a_j\le q$ for 
$1\le j\le k$, and we write $(q,\bfa)$ for the greatest common divisor $(q,a_1,\ldots ,a_k)$ 
of $q$ and $a_1,\ldots ,a_k$. The arcs $\grP(q,\bfa)$ are disjoint, as is easily verified. Let 
$\grP$ denote their union, and put $\grp=[0,1)^k\setminus \grP$.\par

\par We illustrate the subconvexity estimates available through the approach underlying the 
proof of Theorem \ref{theorem1.1} with the following conclusion.

\begin{theorem}\label{theorem1.2}
Let $k$ and $s$ be natural numbers with $k\ge 3$. Suppose that $c_1,\ldots ,c_s$ are 
non-zero integers satisfying the property that 
\begin{equation}\label{1.7}
c_1+\ldots +c_s\ne 0.
\end{equation}
Then, whenever $1\le s<k(k+1)$, one has
\begin{equation}\label{1.8}
\int_{[0,1)^k}f_k(c_1\bfalp;X)\cdots f_k(c_s\bfalp;X)\d\bfalp \ll X^{(s-1)/2+\eps}.
\end{equation}
When $s=k(k+1)$, meanwhile, one has
\begin{equation}\label{1.9}
\int_\grp f_k(c_1\bfalp;X)\cdots f_k(c_s\bfalp;X)\d\bfalp \ll X^{(s-\del)/2+\eps},
\end{equation}
where $\del=1/(4k^3)$, and when $s>k(k+1)$ one has
\begin{equation}\label{1.10}
\int_\grp f_k(c_1\bfalp;X)\cdots f_k(c_s\bfalp;X)\d\bfalp \ll 
X^{s-\tfrac{1}{2}k(k+1)-\tfrac{1}{2}\del+\eps}.
\end{equation}
\end{theorem}

Given the trivial estimate $|f_k(\bfalp;X)|\le 2X+1$, the bounds \eqref{1.8} and \eqref{1.9} 
plainly go beyond those that would result from square-root cancellation, and consequently 
constitute subconvexity estimates in the sense described in our work joint with Br\"udern 
\cite{BW2014}. We remark in this context that, with greater effort, it would be possible to 
establish the estimate \eqref{1.9} with a larger value of $\del$. In this paper we have 
elected to opt for a more concise account yielding reasonable qualitative results, rather than 
seek the strongest quantitative results that might be accessible.\par

We briefly offer a sketch of the strategy underlying the proof of Theorem \ref{theorem1.1}, 
restricting attention to the simpler situation that is the focus of Theorem \ref{theorem1.2}. 
Here, by orthogonality, the mean value
\[
\Ups=\int_{[0,1)^k}\prod_{i=1}^sf_k(c_i\bfalp;X)\d\bfalp 
\]
on the left hand side of \eqref{1.8} counts the number of integral solutions of the system of 
equations
\begin{equation}\label{1.11}
\sum_{i=1}^sc_ix_i^j=0\quad (1\le j\le k),
\end{equation}
with $|x_i|\le X$ $(1\le i\le s)$. For each such solution $\bfx$, and for every integer $y$ 
with $1\le y\le X$, it follows from the binomial theorem that
\[
\sum_{i=1}^sc_i(x_i+y)^j=-c_0y^j\quad (1\le j\le k),
\]
where we write $c_0=-(c_1+\ldots +c_s)$. We note that our hypothesis \eqref{1.7} 
concerning the coefficients $c_i$ ensures that one has $c_0\ne 0$. We therefore see that 
for each integer $y$ with $|y|\le X$, the number of integral solutions of the system 
\eqref{1.11} counted by $\Ups$ is bounded above by the number of integral solutions of
\[
c_0y^j+\sum_{i=1}^sc_iz_i^j=0\quad (1\le j\le k),
\]
with $|z_i|\le 2X$ $(1\le i\le s)$. By averaging over these values of $y$ and invoking 
orthogonality, we thus deduce that
\begin{equation}\label{1.12}
\int_{[0,1)^k}\prod_{i=1}^sf_k(c_i\bfalp;X)\d\bfalp \le X^{-1}
\int_{[0,1)^k}\prod_{j=0}^s f_k(c_j\bfalp;2X)\d\bfalp .
\end{equation}
By comparison with the mean value \eqref{1.8}, we now have an additional variable over 
which to average in \eqref{1.12}, and it is this which permits us to achieve subconvexity. 
We note that, in order to analyse the mean value \eqref{1.9}, which is restricted to minor 
arcs only, we employ some ideas from harmonic analysis previously deployed in our work 
\cite{Woo2012} devoted to the asymptotic formula in Waring's problem.\par

This paper is organised as follows. We derive the fundamental lemma, based on the 
strategy just described, in \S2. This work already permits a swift proof of the first 
subconvex estimate \eqref{1.8} recorded in Theorem \ref{theorem1.2}. In \S3 we begin 
the proof of a more general variant of the minor arc estimate \eqref{1.9} recorded in 
Theorem \ref{theorem1.2}. This lays the foundation of the proof of Theorem 
\ref{theorem1.1}. This preliminary minor arc estimate is converted in \S4 into one more 
accessible to conventional applications of the Hardy-Littlewood method. The major arc 
analysis required to complete the proof of Theorem \ref{theorem1.1} is then tackled in \S5. 
We complete the proofs of Theorems \ref{theorem1.1} and \ref{theorem1.2} in \S6. Finally, 
in \S7, we discuss the non-singularity condition implicit in Theorem \ref{theorem1.1}, 
showing that the existence of non-singular solutions of the system \eqref{1.4} is implied by 
the conditions that we have already noted.\par

Throughout, the letter $\eps$ will denote a positive number. We adopt the convention that 
whenever $\eps$ appears in a statement, either implicitly or explicitly, we assert that the 
statement holds for each $\eps>0$. Our basic parameter will be either $X$ or $B$, a 
sufficiently large positive number. In addition, we use $\ll$ and $\gg$ to denote Vinogradov's 
well-known notation, implicit constants depending at most on $k$, $s$ and $\eps$, as well 
as other ambient parameters apparent from the context. Finally, we define $\|\tet\|$ for 
$\tet\in \dbR$ by putting $\| \tet\|=\min\{|\tet -n|: n\in \dbZ\}$.\vskip.1cm

\noindent {\bf Historical note:} The first version of this paper dates from 2014, motivated 
by the author's proof in January 2014 of the main conjecture in the cubic case of 
Vinogradov's mean value theorem (see \cite{Woo2016a}, which first appeared as 
arXiv:1401.3150). The author is grateful to Julia Brandes, Simon Rydin Myerson, Per 
Salberger and others for their comments on talks on this topic delivered at Warwick, King's 
College London, Oxford and G\"oteborg in the period 2014 to 2016 as the associated ideas 
evolved. These ideas subsequently delivered subconvex conclusions in the Hilbert-Kamke 
problem (see \cite{Woo2022c}) and affine variants of Vinogradov's mean value theorem 
(see \cite{Woo2022a, Woo2022b}, and note also \cite{BH2022}).   

\section{An averaged mean value} We begin by interpreting the strategy outlined at the end 
of the introduction as it applies to a mean value not necessarily open to a Diophantine 
interpretation. This supplies a fairly general conclusion useful in our subsequent 
deliberations. We suppose throughout that $s$, $k$, $\bfy$ and $n$ are fixed as in the 
preamble to the statement of Theorem \ref{theorem1.1}. When $y\in \dbZ$, we define 
$\bet_j(y)=\bet_j(\bfalp;y)$ by putting
\begin{equation}\label{2.1}
\bet_j(\bfalp;y)=y^{k-j}\alp_j\quad (1\le j\le k).
\end{equation}
In the proof of the next lemma as well as in its preamble, when $j\in \{k-1,k\}$, we 
promote concision by abbreviating the differential $\d\alp_1\ldots \d\alp_j$ to $\d\bfalp_j$. 
Then, when $\grB\subseteq \dbR$ is measurable, we introduce the mean value
\begin{equation}\label{2.2}
I_s(\grB;X)=\int_\grB \int_{[0,1)^{k-1}} \prod_{i=1}^s f(c_i\bfbet(y_i))\d\bfalp_k,
\end{equation}
in which $f(\bftet)=f_k(\bftet;X)$ is defined via \eqref{1.6}. Notice that, by orthogonality, 
one has $N_{s,k}(B;\bfy)=I_s([0,1);B)$. We make use of technology associated with 
Vinogradov's mean value theorem. With this in mind, when $t,k\in \dbN$, the parameter 
$X$ is positive, and $\grB\subseteq \dbR$ is measurable, we define
\begin{equation}\label{2.3}
\grJ_{t,k}(\grB;X)=\int_\grB \int_{[0,1)^{k-1}}|f_k(\bfalp;X)|^t\d\bfalp_k.
\end{equation}

\begin{lemma}\label{lemma2.1}
Let $\bfc,\bfy\in (\dbZ\setminus \{0\})^s$, and define
\begin{equation}\label{2.4}
c_0=-(c_1y_1^k+\ldots +c_sy_s^k).
\end{equation}
Suppose that $c_0\ne 0$. Then, whenever $\grB\subseteq \dbR$ is measurable, one has
\[
I_s(\grB;X)\ll X^{-1}(\log X)^{s+1}\prod_{i=0}^s\grJ_{s+1,k}(c_i\grB;X)^{1/(s+1)}.
\]
Here, the constant implicit in Vinogradov's notation may depend on $\bfy$.
\end{lemma}

\begin{proof} We make use of the translation invariance underlying a blown-up version of 
the system of Diophantine equations underlying the mean value \eqref{2.2}. Write
\begin{equation}\label{2.5}
\psi(u;\bftet)=\tet_1u+\ldots +\tet_ku^k.
\end{equation}
Observe first that for each index $i$, and every integral shift $z$, it follows from 
\eqref{1.6} that one has
\begin{equation}\label{2.6}
f(\bfbet(y_i);X)=\sum_{|x-y_iz|\le X}e(\psi(x-y_iz;\bfbet(y_i))).
\end{equation}
Write
\begin{equation}\label{2.7}
\grf_{i,z}(\bfalp;\gam)=\sum_{|x|\le 2X}e \left(\psi(x-y_iz;\bfbet(y_i))+\gam (x-y_iz)\right) .
\end{equation}
In addition, define
\begin{equation}\label{2.8}
K(\gam)=\sum_{|w|\le X}e(-\gam w),
\end{equation}
and put
\[
\Lam=\displaystyle{\min_{1\le i\le s}}|y_i|^{-1}.
\]
Then we deduce from \eqref{2.6} via orthogonality that when $|z|\le \Lam X$, one has
\begin{equation}\label{2.9}
f(\bfbet (y_i);X)=\int_0^1\grf_{i,z}(\bfalp;\gam )K(\gam)\d\gam .
\end{equation}

\par Next, define
\begin{equation}\label{2.10}
\grF_z(\bfalp;\bfgam)=\prod_{i=1}^s\grf_{i,z}(c_i\bfalp;\gam_i).
\end{equation}
Then, on substituting \eqref{2.9} into \eqref{2.2}, we deduce that for each integer $z$ 
satisfying $|z|\le \Lam X$, one has
\begin{equation}\label{2.11}
I_s(\grB;X)=\int_{[0,1)^s}\calI(\bfgam;z)\Ktil(\bfgam)\d\bfgam ,
\end{equation}
where
\begin{equation}\label{2.12}
\calI(\bfgam;z)=\int_\grB \int_{[0,1)^{k-1}}\grF_z(\bfalp;\bfgam)\d\bfalp_k
\end{equation}
and
\begin{equation}\label{2.13}
\Ktil(\bfgam)=\prod_{i=1}^s K(\gam_i).
\end{equation} 

\par By orthogonality, one finds that
\begin{equation}\label{2.14}
\int_{[0,1)^{k-1}}\grF_z(\bfalp;\bfgam)\d\bfalp_{k-1}=
\sum_{|\bfx|\le 2X}\Del(\alp_k,\bfgam,z),
\end{equation}
where $\Del(\tet,\bfgam,z)$ is equal to
\[
e\left( \tet \sum_{i=1}^sc_i(x_i-y_iz)^k+\sum_{i=1}^s (x_i-y_iz)\gam_i\right),
\]
when
\begin{equation}\label{2.15}
\sum_{i=1}^sc_iy_i^{k-j}(x_i-y_iz)^j=0\quad (1\le j\le k-1),
\end{equation}
and otherwise $\Del(\tet,\bfgam,z)$ is equal to $0$.\par

By applying the binomial theorem and recalling \eqref{2.4}, one discerns that whenever the 
system \eqref{2.15} is satisfied by the $s$-tuple $\bfx$, then
\[
c_0z^j+\sum_{i=1}^sc_iy_i^{k-j}x_i^j=0\quad (1\le j\le k-1),
\]
and hence
\[
\sum_{i=1}^sc_i(x_i-y_iz)^k=c_0z^k+\sum_{i=1}^sc_ix_i^k.
\]
Then, on recalling \eqref{2.5}, it follows from \eqref{2.14} that
\[
\int_{[0,1)^{k-1}}\grF_z(\bfalp;\bfgam)\d\bfalp_{k-1}=e(-z\bfgam \cdot \bfy)
\int_{[0,1)^{k-1}}\grF_0(\bfalp;\bfgam)e(c_0\psi(z;\bfalp))\d\bfalp_{k-1}.
\]
From here, we are led from the relation \eqref{2.12} to the formula
\[
\calI(\bfgam;z)=e(-z\bfgam \cdot \bfy)\int_\grB \int_{[0,1)^{k-1}}\grF_0(\bfalp;\bfgam)
e(c_0\psi(z;\bfalp))\d\bfalp_k.
\]
Recalling the notation \eqref{1.6}, we may consequently conclude thus far that
\begin{equation}\label{2.16}
\sum_{|z|\le \Lam X}\calI(\bfgam;z)=\int_\grB \int_{[0,1)^{k-1}}
\grF_0(\bfalp;\bfgam)f(c_0\bfalp-\bfgamtil;\Lam X)\d\bfalp_k,
\end{equation}
where $\bfgamtil$ is defined by putting $\gamtil_1=\bfgam \cdot \bfy$ and $\gamtil_j=0$ 
$(2\le j\le k)$.\par

It is convenient at this point to set $y_0=1$ and to apply orthogonality just as in the 
argument leading to \eqref{2.9}. Thus, on recalling \eqref{2.7}, we see that
\[
f(\bfalp;\Lam X)=\int_0^1\grf_{0,0}(\bfalp;\gam)K_0(\gam)\d \gam,
\]
where
\begin{equation}\label{2.17}
K_0(\gam)=\sum_{|z|\le \Lam X}e(-\gam z). 
\end{equation}
Thus, by applying H\"older's inequality to \eqref{2.16} and recalling \eqref{2.10}, we see 
that
\begin{equation}\label{2.18}
\biggl| \sum_{|z|\le \Lam X}\calI(\bfgam;z)\biggr| \le \int_0^1
\biggl(\prod_{i=0}^s\Ome_i\biggr)^{1/(s+1)}|K_0(\gam_0)|\d\gam_0,
\end{equation}
where
\[
\Ome_0=\int_\grB \int_{[0,1)^{k-1}}|\grf_{0,0}(c_0\bfalp-\bfgamtil;\gam_0)|^{s+1}
\d\bfalp_k
\]
and
\[
\Ome_i=\int_\grB\int_{[0,1)^{k-1}}|\grf_{i,0}(c_i\bfalp;\gam_i)|^{s+1}\d\bfalp_k \quad 
(1\le i\le s).
\]
\par

By a change of variable and application of periodicity modulo $1$, we find from \eqref{2.1} 
and \eqref{2.7} that for $0\le i\le s$, one has
\[
\Ome_i=\int_\grB \int_{[0,1)^{k-1}}|\grf_{i,0}(c_i\bfalp;0)|^{s+1}\d\bfalp_k
=c_i^{-1}\int_{c_i\grB}\int_{[0,1)^{k-1}}|\grf_{i,0}(\bfalp;0)|^{s+1}\d\bfalp_k.
\]
It therefore follows from \eqref{2.3} via orthogonality that 
$\Ome_i=c_i^{-1}\grJ_{s+1,k}(c_i\grB;X)$. Thus we infer from \eqref{2.18} that
\[
\biggl| \sum_{|z|\le \Lam X}\calI(\bfgam;z)\biggr| \le \Biggl( \prod_{i=0}^s\grJ_{s+1,k}
(c_i\grB;X)\Biggr)^{1/(s+1)}\int_0^1|K_0(\gam_0)|\d\gam_0.
\]
On substituting this estimate into \eqref{2.11} and recalling \eqref{2.13}, we therefore 
obtain
\begin{align}
I_s(\grB;X)&\le (2\Lam X)^{-1}\int_{[0,1)^s}\biggl| \sum_{|z|\le \Lam X}
\calI (\bfgam;z)\Ktil(\bfgam)\biggr|\d\bfgam \notag\\
&\ll X^{-1}\prod_{i=0}^s\biggl( \grJ_{s+1,k}(c_i\grB;X)^{1/(s+1)}\int_0^1|K_i(\gam_i)|
\d\gam_i \biggr) ,\label{2.19}
\end{align}
in which we have taken the expedient step of writing $K_i(\gam_i)$ for $K(\gam_i)$ when 
$1\le i\le s$. Recall \eqref{2.8} and \eqref{2.17}. Then the elementary bound 
$K_i(\gam)\ll \min\{X,\|\gam\|^{-1}\}$ shows, as is familiar, that
\[
\int_0^1|K_i(\gam_i)|\d\gam_i\ll \log X\quad (0\le i\le s).
\]
Thus, we conclude from \eqref{2.19} that
\[
I_s(\grB;X)\ll X^{-1}(\log X)^{s+1}\prod_{i=0}^s\grJ_{s+1,k}(c_i\grB;X)^{1/(s+1)}.
\]
Here, we stress that the constant implicit in Vinogradov's notation may depend on $\bfy$. 
This completes the proof of the lemma.
\end{proof}

An almost immediate consequence of Lemma \ref{lemma2.1} delivers the first conclusion of 
Theorem \ref{theorem1.2}.

\begin{lemma}\label{lemma2.2} Let $s$ and $k$ be natural numbers with $1\le s<k(k+1)$. 
Suppose that $\bfc,\bfy\in (\dbZ\setminus \{0\})^s$ and 
$c_1y_1^k+\ldots +c_sy_s^k\ne 0$. Then one has
\[
I_s([0,1);X)\ll X^{(s-1)/2+\eps}.
\]
\end{lemma}

\begin{proof} Put $c_0=-(c_1y_1^k+\ldots +c_sy_s^k)$. We apply Lemma \ref{lemma2.1} 
to obtain the bound
\begin{equation}\label{2.20}
I_s([0,1);X)\ll X^{\eps-1}\prod_{i=0}^s \grJ_{s+1,k}(c_i[0,1);X)^{1/(s+1)}.
\end{equation}
Here, in view of the definition \eqref{2.3}, we have
\[
\grJ_{s+1,k}(c_i[0,1);X)=|c_i|\int_{[0,1)^k}|f_k(\bfalp;X)|^{s+1}\d\bfalp .
\]
Since our hypothesis on $s$ ensures that $s+1\le k(k+1)$, we deduce from the (now 
confirmed) main conjecture in Vinogradov's mean value theorem (for which see 
\cite{BDG2015, Woo2016a, Woo2019}) that
\[
\grJ_{s+1,k}(c_i[0,1);X)\ll X^\eps (X^{(s+1)/2}+X^{s+1-k(k+1)/2}).
\]
By substituting this estimate into \eqref{2.20}, therefore, we conclude that
\[
I_s([0,1);X)\ll X^{\eps}(X^{(s-1)/2}+X^{s-k(k+1)/2}).
\]
The conclusion of the lemma is now immediate.
\end{proof}

In order to obtain the upper bound \eqref{1.8}, we have only to set $y_1=\ldots =y_s=1$ 
to conclude from \eqref{2.2} and Lemma \ref{lemma2.2} that when $1\le s<k(k+1)$, one 
has
\[
\int_{[0,1)^k}f_k(c_1\bfalp;X)\cdots f_k(c_s\bfalp;X)\d\bfalp \ll X^{(s-1)/2+\eps}.
\]

\par In the next lemma, and throughout the remainder of the paper, we suppose that 
$k\ge 3$ and $\bfc,\bfy\in (\dbZ\setminus \{0\})^s$. Moreover, putting 
$c_0=-(c_1y_1^k+\ldots +c_sy_s^k)$, we suppose that $c_0\ne 0$. We next obtain from 
Lemma \ref{lemma2.1} an estimate of minor arc type. When $1\le Q\le X$, we define a 
one-dimensional Hardy-Littlewood dissection as follows. We define the set of major arcs 
$\grM(Q)$ to be the union of the arcs
\[
\grM(q,a)=\{ \alp\in [0,1):|q\alp-a|\le QX^{-k}\},
\]
with $0\le a\le q\le Q$ and $(a,q)=1$, and then write $\grm(Q)=[0,1)\setminus \grM(Q)$ 
for the corresponding set of minor arcs.\par

Next, when $k$ is an integer with $k\ge 2$, we define the exponent $\sig=\sig(k)$ by 
taking
\[
\sig(k)^{-1}=\begin{cases} 2^{k-1},&\text{when $2\le k\le 5$},\\
k(k-1),&\text{when $k\ge 6$}.\end{cases}
\]
Then, when $k\ge 2$ and $1\le Q\le X$, one has
\begin{equation}\label{2.21}
\sup_{\alp_k\in \grm(Q)}\sup_{\bfalp_{k-1}\in [0,1)^{k-1}}|f_k(\bfalp_k;X)|\ll 
X^{1+\eps}Q^{-\sig(k)}.
\end{equation}
The reader may consult \cite[Lemma 2.2]{Woo2022c} for a proof of this conclusion, which 
makes use of the standard literature.

\begin{lemma}\label{lemma2.3}
When $1\le Q\le X$ and $s\ge k(k+1)$, one has
\[
I_s(\grm(Q);X)\ll X^{s-\tfrac{1}{2}k(k+1)+\eps}Q^{-\sig(k)}.
\]
\end{lemma}

\begin{proof} We apply Lemma \ref{lemma2.1} to obtain the bound
\begin{equation}\label{2.22}
I_s(\grm(Q);X)\ll X^{\eps-1}\prod_{i=0}^s\grJ_{s+1,k}(c_i\grm(Q);X)^{1/(s+1)}.
\end{equation}
Here, in view of the definition \eqref{2.3}, we have
\[
\grJ_{s+1,k}(c_i\grm(Q);X)=\int_{c_i\grm(Q)}\int_{[0,1)^{k-1}}
|f_k(\bfalp;X)|^{s+1}\d\bfalp_k.
\]
An elementary exercise confirms that $c_i\grm(Q)\subseteq \grm(Q/|c_i|)\mmod{1}$, and 
hence we deduce from \eqref{2.21} that
\[
\sup_{\alp_k\in c_i\grm(Q)}\sup_{\bfalp_{k-1}\in [0,1)^{k-1}}|f_k(\bfalp_k;X)|\ll 
X^{1+\eps}Q^{-\sig(k)}.
\]
Thus, we find that
\[
\grJ_{s+1,k}(c_i\grm(Q);X)\ll X^{1+\eps}Q^{-\sig(k)}\int_{[0,1)^k}
|f_k(\bfalp;X)|^s\d\bfalp_k.
\]
By applying the (now confirmed) main conjecture in Vinogradov's mean value theorem (see 
\cite{BDG2015, Woo2016a, Woo2019}) once again, we therefore conclude that
\[
\grJ_{s+1,k}(c_i\grm(Q);X)\ll X^{1+\eps}Q^{-\sig(k)}(X^{s/2}+X^{s-k(k+1)/2}).
\]
The conclusion of the lemma follows by substituting this upper bound into \eqref{2.22}.
\end{proof}

The conclusion of Lemma \ref{lemma2.3} is neither quite sufficient, by itself, to deliver the 
bound \eqref{1.9} of Theorem \ref{theorem1.2}, nor the key minor arc input into Theorem 
\ref{theorem1.1}. However, it does provide a bound for the most difficult region of the 
minor arcs. Our goal in \S\S 3 and 4 is to handle the remaining parts of the minor arcs in 
the Hardy-Littlewood dissection.

\section{A generalised minor arc estimate}
Our goal in this section is to lay the foundations for an application of the Hardy-Littlewood 
method capable of delivering the estimate \eqref{1.9} of Theorem \ref{theorem1.2}, as 
well as the conclusion of Theorem \ref{theorem1.1}. To this end we introduce a 
Hardy-Littlewood dissection. First, as a close relative of the mean value $I_s(\grB;X)$ 
introduced in \eqref{2.2}, we define the mean value $T_s(\grA)=T_s(\grA;X)$ for 
measurable sets $\grA\subseteq [0,1)^k$ by writing
\begin{equation}\label{3.1}
T_s(\grA;X)=\int_\grA f(c_1\bfbet(y_1))\cdots f(c_s\bfbet(y_s))\d\bfalp . 
\end{equation}

\par Next, when $1\le Z\le X$, we denote by $\grK(Z)$ the union of the major arcs
\[
\grK(q,\bfa;Z)=\{ \bfalp\in [0,1)^k:\text{$|\alp_j-a_j/q|\le ZX^{-j}$ $(1\le j\le k)$}\},
\]
with $1\le q\le Z$, $0\le a_j\le q$ $(1\le j\le k)$ and $(q,\bfa)=1$, and we define the 
complementary set of minor arcs by putting $\grk(Z)=[0,1)^k\setminus \grK(Z)$. We have 
already defined the one-dimensional Hardy-Littlewood dissection of $[0,1)$ into sets of arcs 
$\grM=\grM(Q)$ and $\grm=\grm(Q)$. We now fix $L=X^{1/(8k^2)}$ and $Q=L^k$, and 
we define $k$-dimensional sets of arcs by taking $\grN=\grK(Q^2)$ and $\grn=\grk(Q^2)$. 
We also need the narrow set of major arcs $\grP=\grK(L)$, and the complementary set of 
minor arcs $\grp=\grk(L)$. In this last dissection, it is convenient to abbreviate 
$\grK(q,\bfa;L)$ to $\grP(q,\bfa)$. We note that this last set of major and minor arcs 
coincide with those defined in the preamble to the statement of Theorem \ref{theorem1.2}.

\par We partition the set of points $(\alp_1,\ldots ,\alp_k)$ lying in $[0,1)^k$ into four 
disjoint subsets, namely
\begin{align*}
\grW_1&=[0,1)^{k-1}\times \grm ,\\
\grW_2&=([0,1)^{k-1}\times \grM)\cap \grn ,\\
\grW_3&=([0,1)^{k-1}\times \grM)\cap (\grN\setminus \grP),\\
\grW_4&=\grP .
\end{align*}
Noting that $\grP\subseteq [0,1)^{k-1}\times \grM$, it follows that 
$[0,1)^k=\grW_1\cup \ldots \cup \grW_4$. Hence, by orthogonality, we infer that
\begin{equation}\label{3.2}
N_{s,k}(X;\bfy)=T_s([0,1)^k)=\sum_{i=1}^4T_s(\grW_i),
\end{equation}
and further that
\begin{equation}\label{3.3}
\int_\grp f(c_1\bfbet(y_1))\cdots f(c_s\bfbet(y_s))\d\bfalp =\sum_{i=1}^3T_s(\grW_i).
\end{equation}

\par The work of \S2 already permits us to announce a satisfactory upper bound for the 
contribution of the set of arcs $\grW_1$ in \eqref{3.2} and \eqref{3.3}.

\begin{lemma}\label{lemma3.1}
When $s\ge k(k+1)$, one has
\[
T_s(\grW_1)\ll X^{s-\tfrac{1}{2}k(k+1)-1/(8k^3)}.
\]
\end{lemma}

\begin{proof} We observe that
\[
T_s(\grW_1)=I_s(\grm(Q);X).
\]
Thus, on substituting $Q=X^{1/(8k)}$ into Lemma \ref{lemma2.3}, noting that 
$\sig(k)>1/k^2$ for $k\ge 3$, the conclusion of the lemma is immediate.
\end{proof}

\section{Further minor arc estimates}
We next estimate the contributions arising from the sets of arcs $\grW_2$ and $\grW_3$ 
within \eqref{3.2} and \eqref{3.3}. We begin with an estimate of Weyl-type for the 
exponential sum $f(c_i\bfbet(y_i))$ $(1\le i\le s)$.

\begin{lemma}\label{lemma4.1}
Suppose that $1\le i\le s$ and $c_iy_i\ne 0$. Then
\[
\sup_{\bfalp\in \grn}|f(c_i\bfbet(y_i))|\ll X^{1-1/(6k^2)}\quad \text{and}\quad 
\sup_{\bfalp \in \grp}|f(c_i\bfbet(y_i))|\ll X^{1-1/(12k^3)}.
\]
\end{lemma}

\begin{proof} We begin by confirming the first bound. Put $\tau=1/(6k^2)$ and 
$\del =1/(5k)$. Since $\tau^{-1}>4k(k-1)$ and $\del>k\tau$, we find from 
\cite[Theorem 1.6]{Woo2012a} that whenever $|f(c_i\bfbet(y_i))|\ge X^{1-\tau}$, there 
exist $q\in \dbN$ and $\bfa\in \dbZ^k$ having the property that
\[
1\le q\le X^\del\qquad \text{and}\qquad |c_iq\bet_j(y_i)-a_j|\le X^{\del-j}\quad 
(1\le j\le k).
\]
Write
\[
r=|c_iy_i^{k-1}|q\quad \text{and}\quad b_j=\frac{a_j|c_iy_i^{k-1}|}{c_iy_i^{k-j}}\quad 
(1\le j\le k).
\]
Then, on recalling from \eqref{2.1} that we have $\bet_j(y_i)=y_i^{k-j}\alp_j$ 
$(1\le j\le k)$, we see that when $X$ is sufficiently large in terms of $\bfy$, one has
\[
|r\alp_j-b_j|\le |y_i|^{j-1}X^{\del-j}\le X^{\del'-j}\quad (1\le j\le k),
\]
in which we have written $\del'=1/(4k)$. In particular, we see that $r\le Q^2$ and 
$|\alp_j-b_j/r|\le Q^2X^{-j}$ $(1\le j\le k)$, and hence $\bfalp\in \grN\mmod{1}$. We 
therefore infer that whenever $X$ is sufficiently large in terms of $k$, and $\bfalp\in \grn$, 
then one must have $|f(c_i\bfbet(y_i))|\ll X^{1-\tau}$ $(1\le i\le s)$, and the first 
conclusion of the lemma follows.\par

In order to confirm the second bound, we put $\tau=1/(12k^3)$ and $\del=1/(10k^2)$. We 
again have $\tau^{-1}>4k(k-1)$ and $\del>k\tau$, and so the same argument applies 
mutatis mutandis. Thus, whenever $|f(c_i\bfbet(y_i))|\ge X^{1-\tau}$, one deduces that 
$\bfalp \in \grP\mmod{1}$. Consequently, when $X$ is sufficiently large in terms of $k$ and 
$\bfy$, and $\bfalp\in \grp$, then one must have $|f(c_i\bfbet(y_i))|\ll X^{1-\tau}$ 
$(1\le i\le s)$. This delivers the second conclusion and completes the proof of the lemma.
\end{proof}

By combining this Weyl-type estimate with the conclusion of Lemma \ref{lemma2.1}, we 
obtain a satisfactory estimate for $T_s(\grW_2)$ by exploiting the observation that 
$\grW_2$ has small measure.

\begin{lemma}\label{lemma4.2}
When $c_iy_i\ne 0$ $(1\le i\le s)$ and $s\ge k(k+1)$, one has
\[
T_s(\grW_2)\ll X^{s-\tfrac{1}{2}k(k+1)-1/(16k)}.
\]
\end{lemma}

\begin{proof} When $\alp_k\in \grM$, define
\[
\grL(\alp_k)=\{ (\alp_1,\ldots ,\alp_{k-1})\in [0,1)^{k-1}: \bfalp\in \grW_2\}.
\]
Then an application of H\"older's inequality leads from \eqref{3.1} to the upper 
bound
\begin{equation}\label{4.1}
T_s(\grW_2)\le \prod_{i=1}^s \biggl( \int_\grM I_i(\alp_k)\d\alp_k \biggr)^{1/s},
\end{equation}
where
\[
I_i(\alp_k)=\int_{\grL(\alp_k)}|f(c_i\bfbet(y_i))|^s\d\alp_1\cdots \d\alp_{k-1}.
\]
Noting that $\grW_2\subseteq \grn$, applying the trivial estimate 
$|f(c_i\bfbet(y_i))|\le 2X+1$, and writing $v=k(k-1)/2$, we deduce that
\begin{equation}\label{4.2}
I_i(\alp_k)\le X^{s-k(k+1)}\Bigl( \sup_{\bfalp\in \grn}|f(c_i\bfbet(y_i))|\Bigr)^{2k}
\int_{[0,1)^{k-1}}|f(c_i\bfbet(y_i))|^{2v}\d\alp_1\cdots \d\alp_{k-1} .
\end{equation}
By orthogonality, the mean value here counts the integral solutions of the system of 
equations
\[
c_iy_i^{k-j}\sum_{l=1}^v(x_l^j-y_l^j)=0\quad (1\le j\le k-1),
\]
with $1\le \bfx,\bfy\le X$, each solution being counted with the unimodular weight
\[
e(c_i\alp_k(x_1^k-y_1^k+\ldots +x_v^k-y_v^k)).
\]
Thus, applying the (now proven) main conjecture in Vinogradov's mean value theorem (see 
\cite{BDG2015, Woo2016a, Woo2019}), we find that one has the bound
\[
\int_{[0,1)^{k-1}}|f(c_i\bfbet(y_i))|^{2v}\d\alp_1\cdots \d\alp_{k-1}\ll 
\grJ_{2v,k-1}([0,1);X)\ll X^{v+\eps},
\]
uniformly in $\alp_k$.\par

Making use of the latter bound, we deduce from \eqref{4.2} via Lemma \ref{lemma4.1} 
that
\[
I_i(\alp_k)\ll X^{s-k(k+1)}\left( X^{1-1/(6k^2)}\right)^{2k}X^{v+\eps}.
\]
Moreover, we have $\text{mes}(\grM)\ll Q^2X^{-k}$. Consequently, we infer that
\begin{align*}
\int_\grM I_i(\alp_k)\d\alp_k&\ll X^{s-\tfrac{1}{2}k(k+1)+\eps}(X^{k-1/(3k)})(Q^2X^{-k})
\\
&\ll X^{s-\tfrac{1}{2}k(k+1)+\eps}(Q^2X^{-1/(3k)}).
\end{align*}
Since $Q^2=X^{1/(4k)}$, we conclude that
\[
\int_\grM I_i(\alp_k)\d\alp_k\ll X^{s-\tfrac{1}{2}k(k+1)-1/(16k)}.
\]
The conclusion of the lemma follows by substituting this bound into \eqref{4.1}.
\end{proof}

The analysis of the set of arcs $\grW_3$ requires standard major arc estimates from the 
theory of Vinogradov's mean value theorem.

\begin{lemma}\label{lemma4.3}
Suppose that $u>\tfrac{1}{2}k(k+1)+2$ and $c_iy_i\ne 0$ for $1\le i\le s$. Then
\[
\int_\grN |f(c_i\bfbet (y_i))|^u\d\bfalp \ll X^{u-k(k+1)/2}.
\]
\end{lemma}

\begin{proof} Suppose that $\bfalp \in \grN$. Then there exist $q\in \dbN$ and 
$\bfa\in \dbZ^k$ for which $(q,\bfa)=1$,
\[
1\le q\le Q^2\quad \text{and}\quad 0\le a_j\le q\quad (1\le j\le k),
\] 
and such that
\[
|\alp_j-a_j/q|\le Q^2X^{-j}\quad (1\le j\le k).
\]
In such circumstances, one has
\[
|c_iy_i^{k-j}\alp_j-c_iy_i^{k-j}a_j/q|\le |c_iy_i^{k-j}|Q^2X^{-j}.
\]
Thus, when $X$ is sufficiently large in terms of $\bfy$, we see that 
$c_i\bfbet(y_i)\in \grK(Q^{2+\eps})\mmod{1}$. Hence, applying periodicity modulo $1$, we 
have
\[
\int_\grN |f(c_i\bfbet (y_i))|^u\d\bfalp \ll \int_{\grK(Q^{2+\eps})}|f(\bfbet)|^u\d \bfbet .
\]
From here we may apply \cite[Lemma 7.1]{Woo2017}, observing that the set of major arcs 
$\grK(Q^{2+\eps})$ is a subset of the major arcs employed in the latter source. Thus, in 
particular, one has
\[
\int_{\grK(Q^{2+\eps})}|f(\bfbet)|^u\d\bfbet \ll X^{u-k(k+1)/2},
\]
and the conclusion of the lemma follows.
\end{proof}

\begin{lemma}\label{lemma4.4}
When $c_iy_i\ne 0$ $(1\le i\le s)$ and $s\ge k(k+1)$, one has
\[
T_s(\grW_3)\ll X^{s-\tfrac{1}{2}k(k+1)-1/(12k^2)}.
\]
\end{lemma}

\begin{proof} An application of H\"older's inequality conveys us from \eqref{3.1} to the 
upper bound
\begin{equation}\label{4.3}
T_s(\grW_3)\le \prod_{i=1}^sJ_i^{1/s},
\end{equation}
where
\[
J_i=\int_{\grW_3}|f(c_i\bfbet(y_i))|^s\d\bfalp .
\]
Since $\grW_3\subseteq \grN\setminus \grP$, we discern from Lemma \ref{lemma4.1} that
\[
\sup_{\bfalp\in \grW_3}|f(c_i\bfbet(y_i))|\le \sup_{\bfalp\in \grp}|f(c_i\bfbet(y_i))|\ll 
X^{1-1/(12k^3)}.
\]
Thus, taking $u=s-k$ and noting that $u\ge \tfrac{1}{2}k(k+1)+3$, it follows from Lemma 
\ref{lemma4.3} that
\begin{align*}
J_i&\le \Bigl( \sup_{\bfalp\in \grp}|f(c_i\bfbet(y_i))|\Bigr)^k\int_\grN 
|f(c_i\bfbet(y_i))|^u\d\bfalp \\
&\ll \bigl( X^{1-1/(12k^3)}\bigr)^k X^{u-k(k+1)/2}.
\end{align*}
The conclusion of the lemma follows on substituting this bound into \eqref{4.3}.
\end{proof}

\section{The analysis of the major arc contribution}
By substituting the conclusions of Lemmata \ref{lemma3.1}, \ref{lemma4.2} and 
\ref{lemma4.4} into the relation \eqref{3.2}, we find that
\begin{equation}\label{5.1}
N_{s,k}(X;\bfy)=T_s(\grP)+O(X^{s-\tfrac{1}{2}k(k+1)-1/(8k^3)}).
\end{equation}
The goal of this section is to obtain an asymptotic formula for $T_s(\grP)$ that suffices to 
confirm \eqref{1.5}, and hence completes the proof of Theorem \ref{theorem1.1}.\par

We begin by introducing the generating functions
\[
I(\bftet;X)=\int_{-X}^X e(\tet_1\gam+\ldots +\tet_k\gam^k)\d\gam 
\]
and
\[
S(q,\bfa)=\sum_{r=1}^q e_q(a_1r+\ldots +a_kr^k),
\]
in which $e_q(u)$ denotes $e^{2\pi iu/q}$. Recall the notation \eqref{2.1}. When 
$1\le i\le s$, we define
\[
I_i(\bftet;X)=I(c_i\bfbet(\bftet;y_i);X)\quad \text{and}\quad 
S_i(q,\bfa)=S(q,c_i\bfbet(\bfa;y_i)).
\]
Then, when $\bfalp\in \grP(q,\bfa)\subseteq \grP$, we write
\[
V_i(\bfalp;q,\bfa)=q^{-1}S_i(q,\bfa)I_i(\bfalp -\bfa/q;X).
\]

\par Define the function $V_i(\bfalp)$ to be $V_i(\bfalp;q,\bfa)$ when 
$\bfalp\in \grP(q,\bfa)\subseteq \grP$, and to be $0$ otherwise. Then, when 
$\bfalp \in \grP(q,\bfa)\subseteq \grP$, we see from \cite[Theorem 7.2]{Vau1997} that
\[
f(c_i\bfbet(\bfalp;y_i))-V_i(\bfalp;q,\bfa)\ll q+X|q\alp_1-a_1|+\ldots +X^k|q\alp_k-a_k|
\ll L^2,
\]
with the implicit constant in Vinogradov's notation depending at most on $c_i$, $y_i$ and 
$k$. Thus, uniformly for $\bfalp\in \grP$, we have the bound
\[
\prod_{i=1}^s f(c_i\bfbet(\bfalp;y_i))-\prod_{i=1}^sV_i(\bfalp) \ll X^{s-1+1/(4k^2)}.
\]
Write
\[
T_s^*(\grP)=\int_\grP \prod_{i=1}^sV_i(\bfalp)\d\bfalp .
\]
Then since $\text{mes}(\grP)\ll L^{2k+1}X^{-k(k+1)/2}$, we deduce that
\begin{equation}\label{5.2}
T_s(\grP)-T_s^*(\grP)\ll X^{s-k(k+1)/2}L^{-1}.
\end{equation}

\par Next write
\[
\Ome (X;D)=[-DX^{-1},DX^{-1}]\times \cdots \times [-DX^{-k},DX^{-k}].
\]
Then one finds that
\begin{equation}\label{5.3}
T_s^*(\grP)=\grS(L)\grT(X;L),
\end{equation}
where
\begin{equation}\label{5.4}
\grT(X;D)=\int_{\Ome(X;D)}\prod_{i=1}^s I_i(\bftet;X)\d\bftet
\end{equation}
and
\[
\grS(D)=\sum_{1\le q\le D}\sum_{\substack{1\le \bfa\le q\\ (q,\bfa)=1}}q^{-s}
\prod_{i=1}^s S_i(q,\bfa) .
\]

\par We examine the truncated singular integral $\grT(X;D)$ and the truncated singular 
series $\grS(D)$ in turn. In this context, we recall the definitions of the real density 
$\sig_\infty$ and the $p$-adic densities $\sig_p$ from the sequel to the statement of 
Theorem \ref{theorem1.1}. We begin by examining the integrals
\[
I(D)=\int_{\Ome(1;D)}\prod_{i=1}^s I_i(\bftet;1) \d\bftet 
\]
and
\[
I_\infty=\int_{\dbR^k}\prod_{i=1}^s I_i(\bftet ;1)\d\bftet .
\]

\begin{lemma}\label{lemma5.1}
Suppose that $s\ge \tfrac{1}{2}k(k+1)+3$. Then the limit 
$I_\infty=\lim_{D\rightarrow \infty}I(D)$ exists, and one has
\[
\grT(X;L)=I_\infty X^{s-k(k+1)/2}+O(X^{s-k(k+1)/2}L^{-1/k}).
\]
Moreover, one has $I_\infty =\sig_\infty$, and provided that the system \eqref{1.4} has a 
non-singular real solution $\bfz\in \dbR^s$, one has $\sig_\infty>0$.
\end{lemma}

\begin{proof} We consider a parameter $D$ with $D\ge 1$, and we put
\[
\Ome^{\rm c}(D)=\dbR^k\setminus \Ome(1;D).
\]
We begin by applying H\"older's inequality to the mean value complementary to \eqref{5.4}, 
obtaining the bound
\begin{equation}\label{5.5}
\int_{\Ome^{\rm c}(D)}\prod_{i=1}^s|I_i(\bftet;1)|\d\bftet \le \prod_{i=1}^s\biggl( 
\int_{\Ome^{\rm c}(D)}|I_i(\bftet;1)|^s\d\bftet \biggr)^{1/s}.
\end{equation}
Recall the bound
\[
I(\bftet;1)\ll (1+|\tet_1|+\ldots +|\tet_k|)^{-1/k},
\]
available from \cite[Theorem 7.3]{Vau1997}. When $\bftet\in \Ome^{\rm c}(D)$, one has 
$|\tet_i|>D$ for some index $i$ with $1\le i\le k$. Then in view of the definition \eqref{2.1} 
of $\bfbet (\bftet;y_i)$, we have
\[
\sup_{\bftet \in \Ome^{\rm c}(D)}|I_i(\bftet;1)|\ll_\bfy D^{-1/k}.
\]
Thus, leaving the dependence of $\bfy$ implicit in Vinogradov's notation henceforth, we 
deduce that
\begin{equation}\label{5.6}
\int_{\Ome^{\rm c}(D)}|I_i(\bftet;1)|^s\d\bftet \ll (D^{-1/k})^{s-\tfrac{1}{2}k(k+1)-2}
\int_{\dbR^k}|I_i(\bftet;1)|^{\tfrac{1}{2}k(k+1)+2}\d\bftet .
\end{equation}

\par When $t>\tfrac{1}{2}k(k+1)+1$, the integral
\[
\int_{\dbR^k}|I(\bftet;1)|^t\d\bftet
\]
converges absolutely (see \cite[Theorem 1.3]{ACK2004}). Thus, one finds by a change of 
variable that
\[
\int_{\dbR^k}|I_i(\bftet;1)|^t\d\bftet \ll \int_{\dbR^k}|I(\bftet ;1)|^t\d\bftet \ll 1.
\]
By substituting this bound first into \eqref{5.6} and then into \eqref{5.5}, we obtain the 
estimate
\[
\int_{\Ome^{\rm c}(D)}\prod_{i=1}^s|I_i(\bftet ;1)|\d\bftet \ll D^{-1/k}.
\]
It therefore follows that the integral $I_\infty$ converges absolutely, and further that one 
has $I_\infty -I(D)\ll D^{-1/k}$. Moreover, by two changes of variable, we are led from 
\eqref{5.4} to the relation
\begin{align}
\grT(X;L)&=X^{s-k(k+1)/2}\int_{\Ome (1;L)}\prod_{i=1}^s I_i(\bftet;1)\d\bftet \notag\\
&=X^{s-k(k+1)/2}\left( I_\infty +O(L^{-1/k})\right) .\label{5.7}
\end{align}

\par At this point we recall the definitions of the quantities $M_\infty (\eta)$ and 
$\sig_\infty$, defined in the sequel to the statement of Theorem \ref{theorem1.1}. Since 
the singular integral $I_\infty$ converges absolutely, it follows from the argument of 
\cite[\S9]{BW2019} that
\[
I_\infty =\lim_{\eta\rightarrow 0+}(2\eta)^{-k}M_\infty (\eta)=\sig_\infty .
\]
It is apparent, moreover, that whenever the system of equations \eqref{1.4} has a 
non-singular real solution, then one has $M_\infty (\eta)\gg \eta^k$, and hence 
$\sig_\infty>0$. In view of the conclusion \eqref{5.7} already obtained, the proof of the 
lemma is complete.
\end{proof}

Before discussing the singular series
\[
\grS=\lim_{D\rightarrow \infty}\grS(D),
\]
we introduce the quantity
\[
A(q)=\sum_{\substack{1\le \bfa\le q\\ (q,\bfa)=1}} q^{-s}\prod_{i=1}^sS_i(q,\bfa).
\]
Thus, one has
\[
\grS(D)=\sum_{1\le q\le D}A(q),
\]
and the singular series is given by the infinite sum
\begin{equation}\label{5.8}
\grS=\sum_{q=1}^\infty A(q).
\end{equation}

\begin{lemma}\label{lemma5.2}
Suppose that $s\ge k(k+1)$. Then the singular series $\grS$ converges absolutely. 
Moreover, for each prime number $p$, the limit $\sig_p$ exists, the product over all primes 
$\prod_p \sig_p$ converges absolutely, and one has $\grS=\prod_p \sig_p$. Moreover, one 
has
\begin{equation}\label{5.9}
\grS(L)=\prod_p \sig_p+O(L^{-1-1/(3k)}),
\end{equation}
and provided that the system \eqref{1.4} has a non-singular $p$-adic solution for each 
prime number $p$, one has $\prod_p \sig_p\gg 1$.
\end{lemma}

\begin{proof} We may suppose that $s\ge k(k+1)$. Put 
$t=\tfrac{1}{2}k(k+1)+\tfrac{5}{2}$. Then, in view of our assumption throughout that 
$k\ge 3$, one sees that
\[
s-t\ge \tfrac{1}{2}k(k+1)-\tfrac{5}{2}\ge k+\tfrac{1}{2},
\]
and in particular $\tfrac{1}{2}k(k+1)+2<t<s$. From \cite[Theorem 7.1]{Vau1997}, we 
find that when $(q,\bfa)=1$ one has $S(q,\bfa)\ll q^{1-1/k+\eps}$. Hence, we deduce 
that
\[
S_i(q,\bfa)\ll_\bfy q^{1-1/k+\eps}.
\]
Again suppressing the implicit dependence on $\bfy$ in Vinogradov's notation, an application 
of H\"older's inequality reveals that
\begin{align*}
A(q)&\le \prod_{i=1}^s \Biggl( \sum_{\substack{1\le \bfa\le q\\ (q,\bfa)=1}}q^{-s}
|S_i(q,\bfa)|^s\Biggr)^{1/s}\\
&\ll q^{\eps-(s-t)/k}\prod_{i=1}^s \Biggl( 
\sum_{\substack{1\le \bfa\le q\\ (q,\bfa)=1}}q^{-t}|S_i(q,\bfa)|^t\Biggr)^{1/s}.
\end{align*}
Since we have arranged parameters so that $s-t\ge k+\tfrac{1}{2}$, we find by means of 
H\"older's inequality that
\begin{equation}\label{5.10}
\sum_{q\ge D}|A(q)|\ll D^{-1-1/(3k)}\prod_{i=1}^s \Biggl( \sum_{q\ge D}
\sum_{\substack{1\le \bfa\le q\\ (q,\bfa)=1}}q^{-t}|S_i(q,\bfa)|^t\Biggr)^{1/s}.
\end{equation}

\par A change of variable supplies the estimate
\[
\sum_{q\ge D}\sum_{\substack{1\le \bfa\le q\\ (q,\bfa)=1}}q^{-t}|S_i(q,\bfa)|^t\ll 
\sum_{q=1}^\infty \sum_{\substack{1\le \bfa\le q\\ (q,\bfa)=1}}q^{-t}|S(q,\bfa)|^t.
\]
By reference to \cite[Theorem 2.4]{ACK2004}, the sum on the right hand side here is 
absolutely convergent for $t>\frac{1}{2}k(k+1)+2$. We therefore derive from \eqref{5.10} 
the upper bound
\begin{equation}\label{5.11}
\sum_{q\ge D}|A(q)|\ll D^{-1-1/(3k)},
\end{equation}
and thus the singular series \eqref{5.8} is absolutely convergent, and one has
\begin{equation}\label{5.12}
\grS-\grS(L)\ll L^{-1-1/(3k)}.
\end{equation}
The standard theory of singular series shows that the function $A(q)$ is a multiplicative 
function of $q$ (see \cite[\S2.6]{Vau1997} for the necessary ideas). Moreover, since 
\eqref{5.11} shows that, for each prime number $p$, one has
\[
\sum_{h\ge H}|A(p^h)|\ll p^{-H(1+1/(3k))},
\]
we see that the limit
\[
\lim_{H\rightarrow \infty}\sum_{h=0}^H A(p^h)
\]
exists, and that the infinite sum
\[
A_p=\sum_{h=0}^\infty A(p^h)
\]
is absolutely convergent with $A_p=1+O(p^{-1-1/(3k)})$. Thus the infinite product 
$\prod_p A_p$ is absolutely convergent and $\grS=\prod_p A_p$. In particular, we deduce 
from \eqref{5.12} that
\begin{equation}\label{5.13}
\grS(L)-\prod_p A_p\ll L^{-1-1/(3k)}.
\end{equation}

\par Once again applying the standard theory of singular series, moreover, one has
\[
\sum_{h=0}^HA(p^h)=p^{H(k-s)}M_p(H),
\]
where $M_p(H)$ denotes the number of solutions of the system
\[
\sum_{i=1}^sc_iy_i^{k-j}z_i^j\equiv 0\mmod{p^H}\quad (1\le j\le k),
\]
with $1\le \bfz\le p^H$. Thus we find that $A_p$ is equal to the $p$-adic density $\sig_p$ 
defined in the sequel to the statement of Theorem \ref{theorem1.1}. We are at liberty to 
assume that the system of equations \eqref{1.4} has a non-singular $p$-adic solution for 
each prime $p$. It therefore follows via Hensel's lemma that there is a non-negative integer 
$\nu_p$ satisfying the property that, whenever $H\ge \nu_p$, one has
\[
M_p(H)\ge p^{(H-\nu_p)(s-k)},
\]
whence
\[
\sig_p=\lim_{H\rightarrow \infty}p^{H(k-s)}M_p(H)\ge p^{-(s-k)\nu_p}>0.
\]
Then, on recalling that $\sig_p=1+O(p^{-1-1/(3k)})$, we find that there is a positive 
integer $p_0$ with the property that
\[
\grS=\prod_p\sig_p\gg \prod_{p>p_0}(1-p^{-1-1/(4k)})\gg 1, 
\]
whilst at the same time $\grS\ll 1$. The proof of the lemma is completed on noting that 
since $A_p=\sig_p$, the relation \eqref{5.13} yields the asymptotic relation \eqref{5.9}.
\end{proof}

We are now equipped to complete the asymptotic analysis of the major arc contribution 
$T_s(\grP)$.

\begin{lemma}\label{lemma5.3}
Suppose that $s\ge k(k+1)$, and the system \eqref{1.4} has a non-singular real solution, 
and a non-singular $p$-adic solution for each prime $p$. Then one has
\[
T_s(\grP)=\sig_\infty \Bigl( \prod_p \sig_p\Bigr) X^{s-k(k+1)/2}+o(X^{s-k(k+1)/2}),
\]
in which the product over real and $p$-adic densities is positive.
\end{lemma}

\begin{proof} By substituting the conclusions of Lemmata \ref{lemma5.1} and 
\ref{lemma5.2} into \eqref{5.3}, we find that
\begin{align*}
T_s^*(\grP)&=\left( \grS+O(L^{-1-1/(3k)})\right) \left( \sig_\infty X^{s-k(k+1)/2}+
O(X^{s-k(k+1)/2}L^{-1/k})\right) \\
&=\sig_\infty \Bigl( \prod_p \sig_p\Bigr) X^{s-k(k+1)/2}+O(X^{s-k(k+1)/2}L^{-1/k}).
\end{align*}
Moreover, the product over real and $p$-adic densities here in the leading asymptotic term 
is positive. We therefore conclude from \eqref{5.2} that
\[
T_s(\grP)=\sig_\infty \Bigl( \prod_p \sig_p\Bigr) X^{s-k(k+1)/2}+
O(X^{s-k(k+1)/2}L^{-1/k}),
\]
and the proof of the lemma is complete.
\end{proof}

\section{The proof of Theorems \ref{theorem1.1} and \ref{theorem1.2}}
The completion of the proofs of our main theorems is now at hand, though we defer to the 
next section a consideration of the nature of the singularities of the system \eqref{1.4}.

\begin{proof}[The proof of Theorem \ref{theorem1.1}] On recalling \eqref{5.1}, we find 
that when $s\ge k(k+1)$ and $c_iy_i\ne 0$ $(1\le i\le s)$, one has
\begin{equation}\label{6.1}
N_{s,k}(X;\bfy)=T_s(\grP)+o(X^{s-k(k+1)/2}).
\end{equation}
The hypotheses of Theorem \ref{theorem1.1} permit us to assume that the system 
\eqref{1.4} possesses non-singular real and $p$-adic solutions, for each prime number $p$. 
Thus, we deduce from Lemma \ref{lemma5.3} that
\[
T_s(\grP)=\calC_{s,k}(\bfy)X^{s-k(k+1)/2}+o(X^{s-k(k+1)/2}),
\]
where $\calC_{s,k}(\bfy)=\sig_\infty \prod_p \sig_p>0$. The conclusion of Theorem 
\ref{theorem1.1} now follows by substituting this asymptotic relation into \eqref{6.1}.
\end{proof}

\begin{proof}[The proof of Theorem \ref{theorem1.2}] The proof of the upper bound 
\eqref{1.8} has already been accomplished in Lemma \ref{lemma2.2} and the discussion 
following the latter. Turning now to the proof of the upper bounds \eqref{1.9} and 
\eqref{1.10}, suppose that $k\ge 3$ and $s\ge k(k+1)$. We set $y_j=1$ for $1\le j\le s$ 
and put $n=c_1+\ldots +c_s\ne 0$. In this scenario, we find that \eqref{3.3} delivers the 
estimate
\begin{equation}\label{6.2}
\int_\grp f_k(c_1\bfalp;X)\cdots f_k(c_s\bfalp;X)\d\bfalp =\sum_{i=1}^3 T_s(\grW_i),
\end{equation}
where, by virtue of Lemmata \ref{lemma3.1}, \ref{lemma4.2} and \ref{lemma4.4},
\begin{equation}\label{6.3}
\sum_{i=1}^3 T_s(\grW_i)\ll X^{s-\tfrac{1}{2}k(k+1)-1/(8k^3)}.
\end{equation}
When $s=k(k+1)$, the right hand side here is $O(X^{(s-\del)/2})$, where $\del=1/(4k^3)$, 
and when $s>k(k+1)$, it is instead $O(X^{s-\tfrac{1}{2}k(k+1)-\tfrac{1}{2}\del})$. In 
either case, therefore, the upper bounds \eqref{1.9} and \eqref{1.10} follow by substituting 
\eqref{6.3} into \eqref{6.2}.
\end{proof}

\section{The non-singularity of non-zero solutions}
Suppose that the system of equations \eqref{1.4} has a non-zero solution 
$\bfz\ne {\mathbf 0}$ lying in either $\dbR^s$ or $\dbQ_p^s$, for a given prime $p$. Our 
goal in this section is to show that this solution is in fact non-singular under the conditions 
discussed in the sequel to the statement of Theorem \ref{theorem1.1}. We assume 
throughout that  the equation \eqref{1.2}, with $n\ne 0$ and $c_i\ne 0$ $(1\le i\le s)$, has 
a solution $\bfy$ with $y_i\ne 0$ $(1\le i\le s)$. Then, should the system \eqref{1.4} have a 
non-zero solution $\bfz$ over $\dbR$, or over $\dbQ_p$, we find that $\bfz$ satisfies the 
system of equations
\begin{equation}\label{7.1}
\sum_{i=1}^s c_iy_i^k\left( z_i/y_i\right)^j=0\quad (1\le j\le k).
\end{equation}
Suppose, by way of deriving a contradiction, that this solution $\bfz$ is singular. Then, for 
any $k$-tuple $(i_1,\ldots ,i_k)$ of natural numbers satisfying 
$1\le i_1<i_2<\ldots <i_k\le s$, one must have
\begin{equation}\label{7.2}
\det \left( jc_{i_l}y_{i_l}^{k-j}z_{i_l}^{j-1}\right)_{1\le j,l\le k}=0.
\end{equation}
Since $c_iy_i\ne 0$ for $1\le i\le s$, a consideration of Vandermonde determinants reveals 
that the condition \eqref{7.2} is satisfied if and only if
\[
0=\det \left( \left( \frac{z_{i_l}}{y_{i_l}}\right)^{j-1}\right)_{1\le j,l\le k}=
\prod_{1\le j<l\le k}\left( \frac{z_{i_l}}{y_{i_l}}-\frac{z_{i_j}}{y_{i_j}}\right) .
\]
This relation implies that
\[
\frac{z_{i_l}}{y_{i_l}}=\frac{z_{i_j}}{y_{i_j}},
\]
for some indices $j$ and $l$ with $1\le j<l\le k$, and thus we are forced to conclude that 
the set $\left\{ z_i/y_i : 1\le i\le s\right\}$ contains at most $k-1$ distinct values.\par

By relabelling indices, we may suppose that, for some integer $r$ with $1\le r\le k-1$, each 
of the rational numbers
\[
\frac{z_1}{y_1},\frac{z_2}{y_2},\ldots ,\frac{z_r}{y_r},
\]
is distinct, and further that, whenever $i>r$, one has
\[
\frac{z_i}{y_i}\in \left\{ \frac{z_1}{y_1},\ldots ,\frac{z_r}{y_r}\right\} .
\]
We define an equivalence relation on indices by defining $i\sim j$ whenever one has 
$z_i/y_i=z_j/y_j$. Then, on putting
\[
C_i=\sum_{\substack{1\le j\le s\\ j\sim i}}c_j\frac{y_j^k}{y_i^k}\quad (1\le i\le r),
\]
we see that the equation \eqref{1.2} becomes
\begin{equation}\label{7.3}
C_1y_1^k+\ldots +C_ry_r^k=n,
\end{equation}
while the equations \eqref{7.1} transform into the new system
\begin{equation}\label{7.4}
\sum_{i=1}^rC_iy_i^k\Big( \frac{z_i}{y_i}\Bigr)^j=0\quad (1\le j\le k),
\end{equation}
subject to the condition
\begin{equation}\label{7.5}
\frac{z_i}{y_i}\ne \frac{z_l}{y_l}\quad (1\le i<l\le r).
\end{equation}
Notice here that since $n\ne 0$, it follows from the equation \eqref{7.3} that 
$C_iy_i^k\ne 0$ for some index $i$ with $1\le i\le r$. Moreover, since 
$(z_1,\ldots ,z_k)\ne {\mathbf 0}$, the relation \eqref{7.5} ensures that $z_i=0$ for at most 
one index $i$ with $1\le i\le r$, and in such circumstances one must have $r\ge 2$.\par

Should the solution $\bfy$ of \eqref{1.2} satisfy the condition that there be no vanishing 
subsums, then $C_i\ne 0$ for $1\le i\le r$. We suppose either that such is the case and 
$\bfz\ne {\mathbf 0}$, or else that $z_i\ne 0$ for $1\le i\le s$. In both circumstances we 
relabel indices in such a manner that $C_iz_i\ne 0$ for $1\le i\le R$, and $C_iz_i=0$ for 
$R<i\le r$. Here, in either scenario, our discussion thus far permits us the assumption that 
$1\le R<k$. We now infer from the system of equations \eqref{7.4} that
\[
\sum_{i=1}^R C_iy_i^k\Bigl( \frac{z_i}{y_i}\Bigr)^j=0\quad (1\le j\le R).
\]
We view these relations as a system of linear equations, with the quantities $C_iy_i^k$ 
$(1\le i\le R)$ as variables. Then since in either scenario under consideration, we have 
$C_iy_i^k\ne 0$ for all indices $i$ with $1\le i\le R$, we see that
\[
\det \biggl( \left( \frac{z_i}{y_i}\right)^j\biggr)_{1\le i,j\le R}=0.
\]
Expanding the Vandermonde determinant, we thus conclude that
\[
\biggl( \prod_{l=1}^R \frac{z_l}{y_l}\biggr) \prod_{1\le i<j\le R}\left( 
\frac{z_i}{y_i}-\frac{z_j}{y_j}\right) =0.
\]
But the hypothesis \eqref{7.5} ensures that the second product on the left hand side is 
non-zero, and our hypothesis $C_iz_i\ne 0$ for $1\le i\le R$ ensures that the first product on 
the left hand side is non-zero. We therefore arrive at a contradiction, so that the solution 
$\bfz$ cannot in fact be singular. The conditions in the sequel to the statement of Theorem 
\ref{theorem1.1} consequently suffice to guarantee the existence of non-singular real and 
$p$-adic solutions, as we had claimed.

\providecommand{\bysame}{\leavevmode\hbox to3em{\hrulefill}\thinspace}


\begin{thebibliography}{18}

\bibitem{ACK2004}
G. I Arkhipov, V. N. Chubarikov and A. A. Karatsuba, \emph{Trigonometric sums in number 
theory and analysis}, De Gruyter Expositions in Mathematics, \textbf{39}, Walter de 
Gruyter, Berlin, 2004. 

\bibitem{Bir1957}
B. J. Birch, \emph{Homogeneous forms of odd degree in a large number of variables}, 
Mathematika \textbf{4} (1957), 102--105.

\bibitem{BDG2015}
J. Bourgain, C. Demeter and L. Guth, \emph{Proof of the main conjecture in Vinogradov's 
mean value theorem for degrees higher than three}, Annals of Math. (2) \textbf{184} 
(2016), no. 2, 633--682.

\bibitem{Bra2014}
J. Brandes, \emph{Forms representing forms and linear spaces on hypersurfaces}, 
Proc. London Math. Soc. (3) \textbf{108} (2014), no. 4, 809--835.

\bibitem{BH2022}
J. Brandes and K. Hughes, \emph{On the inhomogeneous Vinogradov system}, Bull. Aust. 
Math. Soc. \textbf{106} (2022), no. 3, 396--403.

\bibitem{Bra1945}
R. Brauer, \emph{A note on systems of homogeneous equations}, Bull. Amer. Math. Soc. 
\textbf{51} (1945), 749--755.

\bibitem{BW2014}
J. Br\"udern and T. D. Wooley, \emph{Subconvexity for additive equations: pairs of 
undenary cubic forms}, J. Reine Angew. Math. \textbf{696} (2014), 31--67.

\bibitem{BW2019}
J. Br\"udern and T. D. Wooley, \emph{An instance where the major and minor arc integrals 
meet}, Bull. London Math. Soc. \textbf{51} (2019), no. 6, 1113--1128.

\bibitem{BW2022}
J. Br\"udern and T. D. Wooley, \emph{On Waring's problem for larger powers}, submitted, 
28pp; arxiv:2211.10380.

\bibitem{Est1962}
T. Estermann, \emph{A new application of the Hardy-Littlewood-Kloosterman method}, 
Proc. London Math. Soc. (3) \textbf{12} (1962), 425--444.

\bibitem{Par2000}
S. T. Parsell, \emph{The density of rational lines on cubic hypersurfaces}, Trans. Amer. 
Math. Soc. \textbf{352} (2000), no. 11, 5045--5062.

\bibitem{Par2009}
S. T. Parsell, \emph{Asymptotic estimates for rational linear spaces on hypersurfaces}, 
Trans. Amer. Math. Soc. \textbf{361} (2009), no. 6, 2929--2957.

\bibitem{Vau1989a}
R. C. Vaughan, \emph{A new iterative method in Waring's problem}, Acta Math. 
\textbf{162} (1989), no. 1-2, 1--71.

\bibitem{Vau1989b}
R. C. Vaughan, \emph{On Waring's problem for cubes II}, J. London Math. Soc. (2) 
\textbf{39} (1989), no. 2, 205--218.

\bibitem{Vau1997}
R. C. Vaughan, \emph{The Hardy-Littlewood method}, 2nd edn., Cambridge University Press, 
Cambridge, 1997.

\bibitem{VW1994}
R. C. Vaughan and T. D. Wooley, \emph{Further improvements in Waring's problem, II: 
sixth powers}, Duke Math. J. \textbf{76} (1994), no. 3, 683--710.

\bibitem{VW1995}
R. C. Vaughan and T. D. Wooley, \emph{Further improvements in Waring's problem}, 
Acta Math. \textbf{174} (1995), no. 2, 147--240.

\bibitem{Woo2012a}
T. D. Wooley, \emph{Vinogradov's mean value theorem via efficient congruencing}, Ann. of 
Math. (2) \textbf{175} (2012), no. 3, 1575--1627.

\bibitem{Woo2012}
T. D. Wooley, \emph{The asymptotic formula in Waring's problem}, Internat. Math. Res. 
Notices IMRN 2012 (2012), no. 7, 1485--1504.

\bibitem{Woo2017}
T. D. Wooley, \emph{Discrete Fourier restriction via efficient congruencing}, Internat. Math. 
Res. Notices \textbf{2017} (2017), no. 5, 1342--1389.

\bibitem{Woo2016a}
T. D. Wooley, \emph{The cubic case of the main conjecture in Vinogradov's mean value 
theorem}, Adv. Math. \textbf{294} (2016), 532--561.

\bibitem{Woo2016}
T. D. Wooley, \emph{On Waring's problem for intermediate powers}, Acta Arith. 
\textbf{176} (2016), no. 3, 241--247.

\bibitem{Woo2019}
T. D. Wooley, \emph{Nested efficient congruencing and relatives of Vinogradov's mean value 
theorem}, Proc. London Math. Soc. (3) \textbf{118} (2019), no. 4, 942--1016.

\bibitem{Woo2022a}
T. D. Wooley, \emph{Subconvexity in the inhomogeneous cubic Vinogradov system}, J. 
London Math. Soc. (2) \textbf{107} (2023), no. 2, 719--817.

\bibitem{Woo2022b}
T. D. Wooley, \emph{Subconvexity in inhomogeneous Vinogradov systems}, Quart. J. Math. 
\textbf{74} (2023), no. 1, 389--418.

\bibitem{Woo2022c}
T. D. Wooley, \emph{Subconvexity and the Hilbert-Kamke problem}, submitted, 13pp; 
arxiv:2201.02699.

\end{thebibliography}
\end{document}